\documentclass[a4paper,reqno]{amsart}
\usepackage{amssymb}
\usepackage{amsmath}
\usepackage{amsthm}
\usepackage{array}
\usepackage{graphicx}
\usepackage{texdraw}
\usepackage{amsfonts}
\usepackage{hyperref}
\usepackage{caption}
\usepackage{color}
\usepackage{enumitem}

\graphicspath{{./pics/}}

\begingroup
\theoremstyle{plain}
\newtheorem{theorem}{Theorem}[section]
\newtheorem{proposition}[theorem]{Proposition}
\newtheorem{lemma}[theorem]{Lemma}

\theoremstyle{definition}

\newtheorem{remark}[theorem]{Remark}

\endgroup

\newcommand{\ep}{\varepsilon}
\newcommand{\re}{\mathbb R}
\newcommand{\ff}{\varphi}
\renewcommand{\O}{\Omega }
\newcommand{\p}{\partial }

\title[Maximization of the Steklov eigenvalues - diameter constraint]{Maximization of the Steklov eigenvalues with a diameter constraint}
\author{A. Al Sayed, B. Bogosel, A. Henrot, F. Nacry}

\begin{document}

\begin{abstract} 
In this paper, we address the problem of maximizing the Steklov eigenvalues
with a diameter constraint. We provide an estimate of the
Steklov eigenvalues for a convex domain in terms of its diameter and volume and
we show the existence of an optimal convex domain. We establish that balls are never maximizers,
even for the first non-trivial eigenvalue that contrasts with the case of volume or
perimeter constraints. Under an additional regularity assumption, we  are able to prove that the Steklov eigenvalue is multiple for the optimal domain. We illustrate our theoretical results by giving some optimal domains in the plane thanks to a numerical algorithm.
\end{abstract}

\maketitle

\medskip




\section{Introduction}

Among classical questions in spectral geometry lies the problem of minimizing/maximizing under geometric constraints one (or several) eigenvalues of the Laplace operator with various boundary conditions. It has attracted much attention since the first conjecture by Lord Rayleigh stated in this famous book: {\it The Theory of Sound}. In particular, several important open problems have been solved these last twenty years. We refer the reader to the survey \cite{ash}, the monograph \cite{H1} and the recent book \cite{H2} for a good overview on that topic. 

{ In this paper, we deal with the eigenvalue problem for the Laplace operator with Steklov boundary conditions. For a nice survey covering many properties and questions related to these eigenvalues, we refer to \cite{GP1} ({see also \cite[Chapter 5]{H2}}).
Recall that a real $\sigma\geq 0$ is a {\it Steklov eigenvalue} provided that there is $u\in H^1(\O)$ with $u\neq 0$ such that}
\begin{equation}\label{eigen1}
\left\{\begin{array}{ll}
\Delta u = 0 \quad &  \hbox{in }\O,
\\
\frac{\p u}{\p n}= \sigma u \quad &  \hbox{on }\p \O \,.\end{array}
\right.
\end{equation}
{Here and below, $\frac{\p }{\p n}$ stands for the outward normal derivative and $\O$ is a smooth (say Lipschitzian) bounded and open set in $\re^d$. As usual, the problem \eqref{eigen1} is considered in the weak sense, that is}
$$
 \int_\O \nabla u.\nabla v dx= \sigma \int_{\p\O} uv ds\, \quad \forall v\in H^1(\O) .
$$
{In our framework, it is known that the so-called Steklov spectrum is nothing but a discrete sequence satisfying
$$
0=\sigma_0(\O)\leq \sigma_1(\Omega) \leq \sigma_2(\Omega)\leq  \ldots \nearrow +\infty.
$$}
We also point out that each (Steklov) eigenvalue can be computed {through the usual min-max formula:}
\begin{equation}\label{eigen3}
\sigma_k(\O)=\underset{S\in\mathcal{S}_{k+1}}{\min}\,{\underset{v\in S\setminus \{0\}}{\max}\,
\frac{\int_\O |\nabla v|^2 dx}{\int_{\p\O} v^2 ds}} = \underset{v\in [1,u_1,\ldots,u_{k-1}]^\perp}{\min} \,
\frac{\int_\O |\nabla v|^2 dx}{\int_{\p\O} v^2 ds},
\end{equation}
where $\mathcal{S}_{k+1}$ denotes the set of subspaces of dimension $k+1$ of $H^1(\O)$.

\medskip
In this work, we are interested in the maximization of $\sigma_k(\O)$ with a diameter constraint on the set $\O$: 
\begin{equation}
\label{maxD}
\max_{\Omega\in\mathcal{C},D(\Omega)=d_{0}}\sigma_{k}(\Omega),
\end{equation}
for a suitable class $\mathcal{C}$ of open sets in $\mathbb{R}^d$ and where  $D(\Omega)$ denotes the diameter of the open set $\Omega$. Thanks to the  { positive } homogeneity of the Steklov eigenvalues
(i.e., $\sigma_k(t\O)=\sigma_k(\O)/t$ for every real $t>0$) we can replace the problem \eqref{maxD} by any of { the following ones:}
\begin{equation}
\label{maxD1}
\max_{\Omega\in\mathcal{C},D(\Omega)\geq d_{0}}\sigma_{k}(\Omega)
\end{equation}
and
\begin{equation}\label{maxD2}
    \max_{\O\in\mathcal{C}}D(\O)\sigma_{k}(\Omega).
\end{equation}
More precisely, problems \eqref{maxD} and \eqref{maxD1} have the same set of solutions while \eqref{maxD} and \eqref{maxD2}
are equivalent (that is, any solution of one problem is a solution to the other one up to a suitable dilatation).

{The study of the maximization problem for Steklov eigenvalues under a diameter constraint is quite natural in view of the work \cite{BBG}. Indeed, B. Bogosel, D. Bucur and A. Giacomini have established (\cite[Proposition 4.3]{BBG})  an isodiametric control
for Steklov eigenvalues, namely the existence of a positive constant $C(d)$ (depending only on the dimension $d$) such that for every (smooth, bounded and connected) open set $\O\subset \mathbb{R}^d$
\begin{equation}\label{contBBG}
D(\O) \sigma_k(\O) \leq C(d) k^{\frac{2}{d}+1} \quad k=1,2,\ldots
\end{equation}
It should be noted that the problem  \eqref{maxD2} merely reduces in finding the 
optimal upper bound in the estimate \eqref{contBBG}.}

{Besides the later general result, a particular attention has been devoted over the years to the case $k=1$, i.e., to the first (non-trivial) Steklov eigenvalue  $\sigma_1(\O)$.} In 1954, R. Weinstock 
(\cite{Wein}) proved that the disk maximizes $\sigma_1$ among simply connected plane domains of {a given perimeter}.
In fact, for such a maximization problem, the diameter constraint is stronger than the perimeter constraint (itself stronger than the volume constraint)
in the sense that if we {show} that the ball maximizes $\sigma_1$ with a diameter constraint, it would { entail}
that it { also} maximizes $\sigma_1$ with a perimeter constraint and then implies Weinstock's result for simply
connected plane domains. Very surprisingly, we establish in any dimension (see Theorem \ref{notball}) that the ball
is {\bf never a maximizer of $\sigma_1$} under  a diameter constraint.

For the sake of completeness, let us mention that
F. Brock in \cite{Bro}  has proved that the ball in $\mathbb{R}^d$ is always a maximizer of $\sigma_1(\O)$ with
{\it a volume constraint}.  A. Girouard and I. Polterovich {(\cite{GP1})} observed that the disk is not a maximizer in the plane under a {\it perimeter constraint} whenever we remove the simple connectedness assumption: an annulus with a small inner radius provides a better value than the disk.
Nevertheless, recently in \cite{BFNT}, D. Bucur, V. Ferone, C. Nitsch and C. Trombetti extended Weinstock's
result to convex domains in $\re^d$ proving that the ball maximizes $\sigma_1$ with a perimeter constraint
among convex domains. 

{The paper is organized as follows:} in Section \ref{secexi}, we give an estimate of $\sigma_k(\O)$ for a convex domain in terms of its diameter and volume and
we prove existence of an optimal convex domain. Let us mention here that we do not address the question of regularity which seems to be very difficult as it is often the case for such problems.
Assuming regularity of the optimal set, we recall the shape derivative of the Steklov eigenvalue and the shape
derivative of the diameter which will be useful for the numerical simulation provided in Section \ref{secnum} in order
to perform some gradient-type algorithm. Section \ref{secqua} is devoted to qualitative results. 
First, we show that the ball is never a maximizer for $\sigma_k$ with $k=1,2...$. Then, we state and prove (see Theorem \ref{double}) that
a (regular) optimal domain in the {plane} has necessarily a multiple eigenvalue. This is an important result which is suspected to hold for most optimization problems related to eigenvalues. To the best of our knowledge, Theorem \ref{double} provides
the first proof of such a multiplicity property. However, the result still remains a conjecture in other situations (see, e.g., \cite[Open problem 1]{H1}).
At last, we illustrate our theoretical results in Section \ref{secnum}  by giving some optimal domains in the plane thanks to a numerical algorithm.

\section{Existence, optimality conditions}\label{secexi}
\subsection{Existence}
To prove the existence of a maximizer, we will use the classical method of calculus of variations.
Compactness of any class of open sets is almost for free when we work with a diameter constraint, since, by translation invariance, we can assume that our maximizing
sequence lies in a given ball and then, by \cite[Theorem 2.2.25]{HP} we can extract
a subsequence converging with respect to the Hausdorff metric to some open set.
Now, we have to deal with two major issues:
\begin{enumerate}
\item  In general, the diameter is not (sequentially) continuous for the Hausdorff convergence of open sets, see e.g. \cite[Figure 2.4]{HP}.
\item The continuity of Steklov eigenvalues requires additional assumptions as uniform regularity (see \cite{bog1} for the use of the so-called $\varepsilon$-cone property) or a uniform control of norm of the trace operator (see \cite{BGT}). Let us note that we 
can also work in a relaxed setting as in \cite{BBG}.
\end{enumerate}

The two above remarks naturally lead us to work in the setting of convex domains, that is,
$$
\mathcal{C}_d:=\{\Omega\subset \mathbb{R}^d : \Omega\: \mbox{open and convex}, D(\Omega)=d_0\},
$$
where $d_0\geq 0$ is fixed. It is well known (see, e.g., \cite{HP}) that the convexity property is preserved by the Hausdorff convergence. Now, let us assume that a sequence $(\Omega_n)_{n\geq 1}$
of open convex sets of diameter $d_0$ converges to a convex open set $\Omega$
which is nonempty. We are going to prove that $D(\Omega)=d_0$. 

Fix any real number $\varepsilon>0$. Choose two points $x,y\in \Omega$ such that $|x-y|>D(\Omega)-\varepsilon$. 
By virtue of \cite[Proposition 2.2.17]{HP}, we know that the points $x,y\in \Omega_n$ for $n\in \mathbb{N}$ large
enough. Therefore, we see that
$$
\liminf_{n\rightarrow \infty} D(\Omega_n) \geq \liminf_{n\rightarrow \infty} |x-y|\geq D(\Omega)-\varepsilon.
$$
Then, letting $\varepsilon \downarrow 0$ gives the estimate
\begin{equation}
\liminf_{n\rightarrow \infty} D(\Omega_n) \geq D(\Omega).
\label{Lim-Inf}
\end{equation}
Now, fix some increasing function $s:\mathbb{N}\rightarrow\mathbb{N}$
such that
$$
\limsup_{n\rightarrow\infty}D(\Omega_{n})=\lim_{n\rightarrow\infty}D(\Omega_{s(n)}).
$$
For each integer $n\geq1$, let us choose $x_{s(n)},y_{s(n)}\in$
such that $|x_{s(n)}-y_{s(n)}|\geq D(\Omega_{s(n)})-1/s(n)$. Keeping in mind that
$(\Omega_{s(n)})_{n\geq1}$ is a sequence of convex sets contained in a fixed ball
$B$, we can write 
\[
[x_{s(n)},y_{s(n)}]\subset\Omega_{s(n)}\subset B\quad\text{for all}\:n\geq1.
\]
Hence, there is no loss of generality to assume that $x_{s(n)}\rightarrow x$
and $y_{s(n)}\rightarrow y$ for some $x,y\in \mathbb{R}^d$. Since the Hausdorff convergence preserves
the inclusion, we must have $x,y\in\Omega$. Thus, we arrive to the inequality
\begin{equation}
\label{Lim-Sup}
\limsup_{n\rightarrow\infty}D(\Omega_{n})=\lim_{n\rightarrow\infty}D(\Omega_{s(n)})\leq|x-y|\leq D(\Omega).
\end{equation}
It remains to put together \eqref{Lim-Inf} and \eqref{Lim-Sup} to get
$$
\lim_{n\rightarrow\infty}D(\Omega_{n})=D(\Omega).
$$

We are now in position to prove the following result:

\begin{theorem}
For any $k\geq 1$, the problem
$$\max_{\Omega\in \mathcal{C}} \sigma_k(\Omega)$$
has a solution.
\end{theorem}

\begin{proof}
The proof will follow the same lines as in \cite{bog1}. Let $(\Omega_n)_{n\geq 1}$ be a maximizing sequence
(of open convex sets with diameter $d_0$). We recall that, by translation invariance, we can assume that the
sequence $(\Omega_n)_{n\geq 1}$ lies in a given ball and then, by \cite[Theorem 2.2.25]{HP} we can extract
a subsequence converging with respect to the Hausdorff metric to some open set.
There are two possibilities:
\begin{enumerate}
\item there is a subsequence that converges
to a (nonempty) open convex set $\Omega$. Moreover, by the continuity property proved above $D(\Omega)=d_0$;
\item the sequence $(\Omega_n)_{n\geq 1}$ converges to the empty set. This means that it shrinks to a  convex body of dimension at most $d-1$ and $|\Omega_n| \to 0$.
\end{enumerate}

Due to the maximizing property of $(\Omega_n)_{n\geq 1}$, the case $(2)$ cannot occur. In fact, we are going to prove that if $D(\Omega_n) = d_0$ and $|\Omega_n| \to 0$, then 
$\sigma_k(\Omega_n) \to 0$. In order to prove such a convergence result, the following proposition will be needed. It provides an estimate which can be seen a counterpart of the classical estimate of
Steklov eigenvalues by Colbois, El Soufi,
Girouard in terms of the isoperimetric ratio (see \cite{CEG}). For convex domains,
we are able to get a more precise estimate involving volume and diameter.

\begin{proposition}
Let $\Omega$ be a convex domain of diameter $D(\Omega)=D$ in $\mathbb{R}^d$. Then, there exists an explicit constant $C=C(d,k)$ depending
only on the dimension $d$ and on $k$ such that
$$
\sigma_k(\Omega)\leq C \frac{|\Omega|^\frac{1}{d-1}}{D^{\frac{2d-1}{d-1}}}.
$$
\end{proposition}
\begin{proof}[Proof of the proposition]
We will proceed as in the proof of \cite[Proposition 4.2]{bog1}. Let us denote $D:=D(\Omega)$. Pick any diameter $\rho$ of $\Omega$. We associate to it a set $\Omega_0$ (called \emph{region}) which is defined as the part of $\Omega$ contained between two hyperplanes orthogonal to the diameter $\rho$. The width of the region is denoted by $L$.

{\bf Step 1.} Following Part 1 of the proof of \cite[Proposition 4.2]{bog1}, we can get through elementary geometric arguments, the following estimate
\begin{equation}\label{es1}
|\Omega_0| \geq \frac{L^d}{D^d} |\Omega|.
\end{equation}
More precisely, the basic idea to get the latter inequality is to make a comparison with a cone: the smallest
volume for a portion of a cone is near its vertex for which we get exactly this estimate.

{\bf Step 2.} We also need a lower bound of the (lateral) perimeter of the region $\Omega_0$.
To that purpose, we first perform a Steiner symmetrization $\Omega_0^*$ of $\Omega_0$ with respect to the direction of the choosen diameter. This preserves the volume and  
decreases the perimeter. Also, all sections of $\Omega_0^*$ orthogonal to the diameter are $(d-1)$-dimensional balls. Among these ones, pick the one of maximal radius $r$. 
Obviously, the cylinder of radius $r$ and height $L$ contains the region $\Omega_0^*$, so its volume given by $\omega_{d-1} Lr^{d-1}$ is greater than $|\Omega_0^*|$. Here and below, 
$\omega_k$ denotes the volume of the unit ball in $\mathbb{R}^k$.
Using \eqref{es1}, this allows us to obtain a lower bound for $r$
$$
r^{d-1} \geq  \frac{L^{d-1}|\Omega|}{\omega_{d-1} D^d},
$$
in particular,
\begin{equation}\label{es2}
r \geq  L \left(\frac{|\Omega|}{\omega_{d-1}D^d}\right)^{1/(d-1)}.
\end{equation}

On the other hand, note that we can always include in $\Omega_0^*$ two cones with basis balls of radius $r$ and heights which sum up to $L$. The perimeter of convex sets is monotone with respect to 
inclusion (see, e.g., \cite[Lemma 2.2.2.]{bubu}) therefore the (lateral) perimeter of 
$\Omega_0^*$ can be bounded from below by the sum of the ones for the two cones, and a lower bound of the following form can be found:
\[ 
P(\Omega_0^*) \geq \frac{\omega_{d-2}}{d-1}\, L r^{d-2}.
 \]
Using  \eqref{es2} we arrive to
\begin{equation}\label{es3}
P(\Omega_0) \geq \frac{\omega_{d-2}}{d-1}\, L^{d-1}
\left(\frac{|\Omega|}{\omega_{d-1} D^d}\right)^{(d-2)/(d-1)}.
\end{equation}

{\bf Step 3.} Finally we obtain an upper bound for the Steklov eigenvalues by using
the min-max formula \eqref{eigen3}. Assume that the diameter is in the direction of the first coordinate $x_1$.
Let us divide the diameter $D$ of $\Omega$
into $k+1$ equal parts and build a test function $u_i$, depending only on $x_1$ in the region defined by each one of these segments. 

In a segment $S_i$ of length $D/(k+1)$ consider $a = D/(4(k+1))$ and define 
the function $u_i$ piecewise affine as follows: 
\begin{itemize}
	\item on the segment of length $2a$ whose midpoint coincides with the middle
	of $S_i$ define $u_i=1$.
	\item on the outer segments of length $a$ let the function $u_i$  goes to zero with gradient $1/a$.
\end{itemize}

Now, let us estimate the Rayleigh quotient associated to $u_i$:
\begin{itemize}
	\item $\int_\Omega |\nabla u_i|^2 \leq \frac{1}{a^2}|\Omega|$.
	\item $\int_{\partial \Omega} u_i^2 \geq P(\{u_i=1\})$. Since the set $\{u_i=1\}$ is a \emph{region} of width $2a$, we can use the estimate \eqref{es3} to get
	\[
	 P(\{u_i=1\}) \geq C_d (2a)^{d-1}\left( \frac{|\Omega|}{D^d}\right)^\frac{d-2}{d-1},
	\]
\end{itemize}
with $C_d:=\omega_{d-2}/(d-1)\omega_{d-1}^{(d-2)/(d-1)}$.
Therefore, we have
\[ 
\frac{\int_\Omega |\nabla u_i|^2}{\int_{\partial \Omega} u_i^2} \leq 
\frac{[2(k+1)]^{d+1} |\Omega|^{\frac{1}{d-1}}}{4C_d D^{\frac{2d-1}{d-1}}}.
 \]
Since we can construct $k+1$ such functions with disjoint supports in $\Omega$, we conclude that this also gives an upper bound for $\sigma_k(\Omega)$. 
\end{proof}

Now, let us come back to the proof of the existence result. We have established (thanks to the latter
proposition) that the maximizing sequence $(\Omega_n)_{n\geq 1}$ converges (up to a subsequence) with respect to the Hausdorff distance to some open set $\Omega$. Since the convexity and the diameter are
preserved, $\Omega$ belongs to the class $\mathcal{C}$. Let $B$ be a (compact) ball
included in $\Omega$. By \cite[Proposition 2.2.17]{HP}, $B$ is also included into
$\Omega_n$ for $n$ large enough. By \cite[Proposition 2.4.4]{HP}, all the sets
$\Omega_n$ and $\Omega$ satisfy the $\varepsilon$-cone property with the same constant
$\varepsilon$ (related to this ball $B$). Moreover, we also have (due to the convexity) the convergence of the involved perimeters, i.e.,
$P(\Omega_n)\to P(\Omega)$.
Thus, by \cite[Theorem 3.5]{bog1}, we have $\sigma_k(\Omega_n) \to \sigma_k(\Omega)$
and the existence follows.
\end{proof}

\medskip\noindent
\begin{remark}
The diameter constraint is, in some sense, more flexible that the volume or the perimeter
constraint. Let us illustrate this by considering a domain $\Omega$ with holes (that is, the complement $\mathbb{R}^d\setminus \Omega$ is disconnected). Filling those holes does not modify the diameter but it would {\bf increase} the associated Steklov eigenvalues (this can be seen through the Rayleigh quotient of any test function: the numerator will increase while the denominator will decrease). As a consequence, there is no loss of generality to state the maximization problem on the class of domains without holes (i.e., simply connected domains in the plane). Nevertheless, the existence
of a maximizer is far being clear in such a class.
\end{remark}

\subsection{Derivative of Steklov eigenvalues}
We are interested in writing optimality conditions for our maximization problem involving a diameter constraint. For that purpose, we use the classical notion of
shape derivative (see, e.g., \cite[Chapter 5]{HP} for more details on that concept). The theorem below gives the formulae for the shape derivative of Steklov eigenvalues. It is a particular case of a more general result which appears in the paper by Dambrine, Kateb, Lamboley \cite{DKL} devoted to the so-called
Wentzell operator and its eigenvalues.

\begin{theorem}\label{derisigma}  Let $\Omega$ be a nonempty open bounded set of class $C^3$. The following hold for any $V\in W^{3,\infty}(\Omega,\mathbb{R}^{d})$.\\
$(a)$ If $\sigma_{k}:=\sigma_{k}(\Omega)$ is a simple eigenvalue
of the Steklov problem, then the application $t\mapsto\sigma_{k}(t):=\sigma_{k}(\Omega_{t})$
(where as usal $\Omega_{t}:=(I+tV)(\Omega))$ is differentiable
and the derivative at $0$ is
\[
(\sigma_{k})'(0)=\int_{\partial\Omega}(\left|\nabla_{\tau}u\right|^{2}-\left|\frac{\partial u}{\partial n}\right|^{2}-\sigma_{k}H\left|u\right|^{2})V.n,
\]
where $u(\cdot)$ is the normalized (Steklov) eigenfunction associated
to $\sigma_{k}$.\\
$(b)$ Let $(u_{k})_{1\leq k\leq m}$ be the family of (Steklov) eigenfunctions
associated to a multiple eigenvalue $\sigma$ of order $m\geq2$.
Then, there exists $m$ functions $t\mapsto\sigma_{k}(t)$ defined
in a neighborhood of $0$ such that
\begin{enumerate}
\item $\sigma_{k}(0)=\sigma_{k}$;
\item For every $t$ near $0$, $\sigma_{k}(t)$ is an eigenvalue of $\Omega_{t}:=(I+tV)(\Omega)$;
\item The functions $t\mapsto\sigma_{k}(t)$ admit derivatives and their
values at $0$ are eigenvalues of the matrix $M=(M_{i,j})_{1\leq i,j\leq m}$
defined by
\[
M_{i,j}=\int_{\partial\Omega}(\nabla_{\tau}u_{i}.\nabla_{\tau}u_{j}-\frac{\partial u_{i}}{\partial n}\frac{\partial u_{j}}{\partial n}-\sigma Hu_{i}u_{j})V.n
\]
\end{enumerate}
\end{theorem}

\subsection{Shape derivatives of the diameter}\label{secderdiam}

The development of optimality conditions for our maximization problem also requires the shape derivative of the diameter. This is the aim of what follows. 

We work here in the context of a general (real) normed space $(X,\left\Vert \cdot\right\Vert )$. Let us consider a multimapping (i.e., a set-valued mapping) $C:I\rightrightarrows X$ with bounded values
defined on a real interval $I:=[T_0,T]$ with $T_0<T$. We introduce the
function $\delta:I\rightarrow\mathbb{R}$ defined by
\[
\delta(t):=D(C(t)):=\sup_{(x,y)\in C(t)^{2}}\left\Vert x-y\right\Vert \quad\text{for all}\:t\in I.
\]
Observe first that the function $\delta(\cdot)$ has no differentiability properties in
general since for a given function $f:I\rightarrow\mathbb{R}_{+}$,
we obviously have
\[
\delta(t)=D([0,f(t)])=f(t)\quad\text{for all}\:t\in I.
\]
This leads to require some regularity  assumptions on the multimapping $C(\cdot)$. Assume that $C(\cdot)$ is $\gamma$-Lipschitz relative to the Hausdorff
distance for some real $\gamma\geq0$, i.e.,
\[
C(t)\subset C(s)+\gamma\left|t-s\right|\mathbb{B}\quad\text{for all}\:s,t\in I,
\]
where $\mathbb{B}$ stands for the closed unit ball of $(X,\left\Vert \cdot\right\Vert )$.
Such an hypothesis entails for every $s,t\in I$,
\begin{flalign*}
\delta(t)=\sup_{x,y\in C(t)}\left\Vert x-y\right\Vert  & \leq\sup_{x,y\in C(s)+\gamma\left|t-s\right|\mathbb{B}}\left\Vert x-y\right\Vert \\
 & =\sup_{x,y\in C(s),b_{1},b_{2}\in\mathbb{B}}\left\Vert x-y+\gamma\left|t-s\right|(b_{1}-b_{2})\right\Vert \\
 & \leq2\gamma\left|t-s\right|+\sup_{x,y\in C(s)}\left\Vert x-y\right\Vert \\
 & =2\gamma\left|t-s\right|+\delta(s),
\end{flalign*}
hence the mapping $\delta(\cdot)$ is $2\gamma$-Lipschitz
continuous on $I$. In particular, if $\mathrm{dim}\,X<\infty$, Rademacher's theorem says that $\delta(\cdot)$ is almost everywhere differentiable
on $I$.\\

Coming back to shape optimization, we are going to assume that  $0\in \mathrm{int}\,I$ along with
\[
C(t):=\left\{ x+tV(x):x\in \Omega\right\}=(\mathrm{Id}_X+tV)(\Omega)=:\Omega_t \quad\text{for all}\:t\in I,
\]
where $\mathrm{Id}_X$ denotes the identity mapping on $X$ and where $\Omega$ is a given nonempty relatively compact subset of $X$ and $V:\Omega\rightarrow X$
is a bounded continuous mapping. Writing 
$$
x+tV(x)=x+sV(x)+(t-s)V(x) \quad \text{for all}\:x\in \Omega,\: \text{all}\: t,s\in I
$$
we then see
\[
\Omega_t\subset \Omega_s+\sup_{x\in \Omega}\left\Vert V(x)\right\Vert \left|t-s\right|\mathbb{B}.
\]
According to what precedes, we know that the diameter $\delta(\cdot)$ is differentiable  almost everywhere on $I$ whenever $X=\mathbb{R}^d$. Besides the latter differentiability property, we are going to establish the existence of the one-sided limit
\[
\lim_{t\downarrow0}\frac{\delta(t)-\delta(0)}{t}.
\]
This amounts to say that $D(\cdot)$ has a shape derivative in the direction $V$. Let us introduce the set
of {\it diameter points} of the set $\Omega$:
\[
\mathcal{D}_{\Omega}:=\{ (x,y)\in\overline{\Omega}:\left\Vert x-y\right\Vert=D(\Omega) \}.
\]
As usual, here and below, $\overline{\Omega}$ denotes the closure of $\Omega$ in $X$. First, note that
\[
\overline{C(t)}=\left\{ x+tV(x):x\in\overline{\Omega}\right\} \quad\text{for all}\:t\in I.
\]
Fix any $(x_{0},y_{0})\in \mathcal{D}_{\Omega}$. We obviously have for any $t\in I$,
\begin{flalign*}
\frac{1}{2}\big(\delta^{2}(t)-\delta^{2}(0)\big) & \geq\frac{1}{2}\left\Vert x_{0}+tV(x_{0})-\big(y_{0}+tV(y_{0})\big)\right\Vert ^{2}-\frac{1}{2}\left\Vert x_{0}-y_{0}\right\Vert ^{2}\\
 & =t\left\langle x_{0}-y_{0},V(x_{0})-V(y_{0})\right\rangle +\frac{t^{2}}{2}\left\Vert V(x_{0})-V(y_{0})\right\Vert ^{2},
\end{flalign*}
hence
\[
\liminf_{t\downarrow0}\frac{\delta^{2}(t)-\delta^{2}(0)}{2t}\geq\left\langle x_{0}-y_{0},V(x_{0})-V(y_{0})\right\rangle .
\]
Since $(x_{0},y_{0})$ has been arbitrarily chosen in the set $\mathcal{D}_{\Omega}$,
we get
\[
\liminf_{t\downarrow0}\frac{\delta^{2}(t)-\delta^{2}(0)}{2t}\geq\sup_{(x,y)\in\mathcal{D}_{\Omega}}\left\langle x-y,V(x)-V(y)\right\rangle .
\]
Now, let $(t_{n})_{n\geq1}$ be a sequence of positive real numbers such
that $t_{n}\rightarrow0$ and
\[
\limsup_{t\downarrow0}\frac{\delta^{2}(t)-\delta^{2}(0)}{2t}=\lim_{n\rightarrow\infty}\frac{\delta^{2}(t_{n})-\delta^{2}(0)}{2t_{n}}.
\]
For each integer $n\geq1$, pick any $x_{n},y_{n}\in C(t_{n})$ such
that
\[
\delta^{2}(t_{n})-t_{n}^{2}<\left\Vert x_{n}-y_{n}\right\Vert ^{2}\leq\delta^{2}(t_{n}).
\]
According to the definition of $C(\cdot)$, we may write for every
integer $n\geq1$, $x_{n}=u_{n}+t_{n}V(u_{n})$ and $y_{n}=v_{n}+t_{n}V(v_{n})$
for some $u_{n},v_{n}\in \Omega$. From the compactness of $\overline{\Omega}$,
we may suppose without loss of generality that $u_{n}\rightarrow u$
and $v_{n}\rightarrow v$ for some $u,v\in\overline{\Omega}$.  It is straightforward to check that
\begin{flalign*}
\delta^{2}(t_n)-\delta^{2}(0) & <\left\Vert x_{n}-y_{n}\right\Vert ^{2}+t_{n}^{2}-\left\Vert u_{n}-v_{n}\right\Vert ^{2}\\
 & <2t_{n}\left\langle u_{n}-v_{n},V(u_{n})-V(v_{n})\right\rangle +t_{n}^{2}\left\Vert V(u_{n})-V(v_{n})\right\Vert ^{2}+t_{n}^{2},
\end{flalign*}
in particular
\[
\limsup_{t\downarrow0}\frac{\delta^{2}(t)-\delta^{2}(0)}{2t}\leq\left\langle u-v,V(u)-V(v)\right\rangle.
\]
We claim that $(u,v)\in\mathcal{D}_{\Omega}$. Indeed, we have
\[
\delta^{2}(t_{n})\geq\left\Vert x-y+t_{n}\big(V(x)-V(y)\big)\right\Vert ^{2} \quad \text{for all}\:(x,y)\in \mathcal{D}_{\Omega},\:\text{all}\:n\geq 1,
\]
which gives the inequality
$$
\liminf_{n\rightarrow\infty}\delta^{2}(t_{n})\geq\delta^{2}(0)=D(\Omega)^2
$$
and
\[
\delta^{2}(0)=\liminf_{n\rightarrow\infty}(\delta^{2}(t_{n})-t_{n}^{2})\leq\liminf_{n\rightarrow\infty}\left\Vert x_{n}-y_{n}\right\Vert ^{2}=\lim_{n\rightarrow\infty}\left\Vert u_{n}-v_{n}+t_{n}(V(u_{n})-V(v_{n}))\right\Vert ^{2}=\left\Vert u-v\right\Vert ^{2}.
\]
Putting what precedes together, we arrive to
\[
\sup_{(x,y)\in\mathcal{D}_{\Omega}}\left\langle x-y,V(x)-V(y)\right\rangle \leq\liminf_{t\downarrow0}\frac{\delta^{2}(t)-\delta^{2}(0)}{2t}\leq\limsup_{t\downarrow0}\frac{\delta^{2}(t)-\delta^{2}(0)}{2t}\leq\left\langle u-v,V(u)-V(v)\right\rangle 
\]
Consequently, the function $\frac{1}{2}\delta^{2}(\cdot)$ has a right
derivative at $0$ given by 
\[
\lim_{t\downarrow0}\frac{\delta^{2}(t)-\delta^{2}(0)}{2t}=\sup_{(x,y)\in\mathcal{D}_{\Omega}}\left\langle x-y,V(x)-V(y)\right\rangle .
\]

We summarize those features in the following proposition.
\begin{proposition} Let $\Omega$ be a nonempty open relatively compact subset of a real normed
space $(X,\|\cdot\|)$ and let $V:\Omega\rightarrow X$ be a bounded and continuous
mapping. Then, one has
\begin{equation}\label{derdiam}
\lim_{t\downarrow0}\frac{D(\Omega_{t})-D(\Omega)}{t}=\frac{1}{D(\Omega)}\sup_{(x,y)\in\mathcal{D}_{\Omega}}\left\langle x-y,V(x)-V(y)\right\rangle,
\end{equation}
where $\mathcal{D}_{\Omega}:=\left\{ (x,y)\in\overline{\Omega}:\left\Vert x-y\right\Vert =D(\Omega)\right\} $
and $\Omega_{t}:=\left\{x+tV(x):x\in \Omega\right\} $ for every $t>0$.
\end{proposition}

\section{Qualitative properties}\label{secqua}
\subsection{Case of the ball}

As recalled in the introduction, the ball maximizes the quantity $\sigma_1(\O)$ with a volume constraint (see
\cite{Bro}). It also maximizes $\sigma_1(\Omega)$ with a perimeter constraint among planar
simply connected domains (\cite{Wein}) and in any dimension among convex domains (\cite{BFNT}).
Therefore, it is quite natural to expect that the ball is also a maximizer (or at least a local maximizer) for $\sigma_1$ in the setting of a diameter constraint. The following theorem shows that it is not the case!

\begin{theorem}\label{notball}
The ball is not a local maximizer for Problem \eqref{maxD2} for $\sigma_1$ in any dimension.
\end{theorem}
\begin{proof}
The idea of the proof is simply to find a perturbation of the ball which increases the product $D(B)\sigma_1(B)$. Without loss of generality, we work with the unit ball and
we use the usual spherical coordinates, that is,
$$\left\lbrace
\begin{array}{l}
x_1=\cos\ff_1\\
x_2=\sin\ff_1 \cos\ff_2 \\
\vdots \\
x_{d-1}=\sin\ff_1 \sin\ff_2 \ldots \sin\ff_{d-2} \cos\ff_{d-1} \\
x_{d}=\sin\ff_1 \sin\ff_2 \ldots \sin\ff_{d-2} \sin\ff_{d-1} \\
\end{array}
\right.$$
where $\ff_j\in [0,\pi]$ for $j\leq d-2$ while $\ff_{d-1}\in [0,2\pi]$. At last, the area element
is given by
$$
ds= \sin^{d-2}(\ff_1)\sin^{d-1}(\ff_2)\ldots \sin(\ff_{d-2}) d\ff_1 d\ff_2 \ldots d\ff_{d-1}.
$$

Let us consider a perturbation driven by the vector field defined in a neighbourhood of the unit
sphere by 
$$
V(X)=\big(a_2\cos(2\ff_{d-1}) + a_4\cos(4\ff_{d-1})\big)X,
$$
for some positive coefficients $a_2,a_4$ which will be chosen later. This means that, for every $\ep>0$ small enough,
we consider some perturbations of the unit ball $B$ defined by $B_\ep:=\{X+\ep V(X), X\in B\}$.
We have recalled in Theorem \ref{derisigma} (see also \cite[Corollary 3.8]{DKL}) that
the first Steklov eigenvalue $\sigma_1(B_\ep)$ has a directional derivative (even if this eigenvalue
is multiple). Such directional derivatives are given by the eigenvalues of the $d\times d$
matrix $\mathcal{M}$ whose entries are (as usual $\omega_d$ denotes the volume of the unit ball while
$\delta_{jk}$ is the Kronecker symbol)
$$
\mathcal{M}_{j,k}=\frac{\delta_{jk}}{\omega_d} \int_{\p B} V.n ds - \frac{d+1}{\omega_d} \int_{\p B} x_j x_k V.n ds\,.
$$
In our case, since $n=X$ on the unit sphere, we have $V.n=a_2\cos(2\ff_{d-1}) + a_4\cos(4\ff_{d-1})$. It is then not difficult to check that all the coefficients of the above matrix are zero except 
$\mathcal{M}_{{d-1},{d-1}}$ and $\mathcal{M}_{d,d}$ which are respectively given by
$$\mathcal{M}_{{d-1},{d-1}}=-\frac{d+1}{\omega_d} \int_{\p B} \prod_{j=1}^{d-2} \sin^2(\ff_j)\cos^2(\ff_{d-1})
\big(a_2\cos(2\ff_{d-1}) + a_4\cos(4\ff_{d-1})\big) ds$$
$$\mathcal{M}_{d,d}=-\frac{d+1}{\omega_d} \int_{\p B} \prod_{j=1}^{d-2} \sin^2(\ff_j) \sin^2(\ff_{d-1})
\big(a_2\cos(2\ff_{d-1}) + a_4\cos(4\ff_{d-1})\big) ds$$
Let us denote by $j_p:=\int_0^\pi \sin^p t\,dt$ (twice the classical Wallis' integrals). The previous formulae
can be rewritten as
$$
\mathcal{M}_{{d-1},{d-1}}=-\frac{d+1}{\omega_d} \prod_{p=3}^d j_p a_2 \frac{\pi}{2} \quad \text{and} \quad \mathcal{M}_{{d},{d}}=\frac{d+1}{\omega_d} \prod_{p=3}^d j_p a_2 \frac{\pi}{2}.
$$
Therefore, the eigenvalues of $\mathcal{M}$ are 0 of order $d-2$, $-Ka_2$ and $Ka_2$ where
$K$ is the positive constant (which is explicitly computable) $K:= \frac{(d-1)\pi}{2\omega_d}\prod_{p=3}^d j_p$.
In other words, the smallest eigenvalue $\sigma_1(B_\ep)$ has the following expansion (keep in mind that $a_2>0$)
$$
\sigma_1(B_\ep)=1-\ep K a_2 +o(\ep).
$$
Now, let us introduce the two antipodal points $N=(0,0,\dots, 1)$ and $S=(0,0,\ldots, -1)$. Through the deformation
they are sent to 
$$N_\ep=(0,0,\ldots,1+\ep(a_2+a_4))\quad \text{and}\quad S_\ep=(0,0,\ldots,-1-\ep(a_2+a_4))\,.$$
Thus, the diameter of $B_\ep$ is greater than  $N_\ep S_\ep =2+2\ep (a_2+a_4)$ and then  
$$
D(B_\ep)\sigma_1(\ep) \geq 2\big(1+\ep (a_2+a_4)\big)\big(1-\ep K a_2 +o(\ep)\big)=2\big(1+\ep(a_4-(K-1)a_2)+o(\ep)\big).
$$
It remains to choose the coefficient $a_4$ such that $a_4>(K-1)a_2$ to get the claim.
\end{proof}

Let us point out here that the main idea of the latter proof is quite elementary. Indeed, we construct a perturbation with two trigonometric
terms, the first one seen by the eigenvalue and the other one only seen by the diameter. Then, a suitable combination allows us to get a positive first derivative. In a same way, we are able to extend Theorem \ref{notball} to any Steklov eigenvalue. The details of the proof are left to the reader.

\begin{theorem}\label{notball2}
The ball is not a maximizer for Problem \eqref{maxD2} for any $\sigma_k$ and any dimension.
\end{theorem}
\subsection{Multiplicity}

As explained in the introduction, it is suspected that most of optimization problems for eigenvalues
have solution with multiplicity. For example, if we denote by $\O_k^*$ a domain
which minimizes the $k$-th eigenvalue of the Laplacian with Dirichlet boundary conditions (see \cite{Buc}
and \cite{Maz-Pra}) it is still an open problem (see, e.g., \cite{H1}) to prove that $\lambda_{k-1}(\O_k^*)=\lambda_{k}(\O_k^*)$
for any $k\geq 3$ (what we know so far is the case $k=2$).

In our context, we are able to prove (by contradiction) such a result for a smooth optimal domain in the plane. We establish it in both situations: without any constraint or with a convexity constraint.

\begin{theorem}\label{double} Let $\Omega^{\star}$ be a smooth ($C^3$) optimal domain in the plane
with or without convexity constraints. Then, the $k$-th Steklov
eigenvalue $\sigma_{k}(\Omega^{\star})$ is multiple. 
\end{theorem}

\begin{proof}
Since there is no ambiguity here, we will denote by $\sigma$ the Steklov eigenvalue and $u$ a normalized
associated Steklov eigenfunction.
Let us start with the unconstrained case. First, note that the smoothness
assumption on the optimal domain implies two things:
\begin{itemize}
\item we can use the shape derivative formulae stated in Theorem \ref{derisigma};
\item by elliptic regularity, the eigenfunctions are at least $C^2$ up to the boundary.
\end{itemize}
We write the usual Rellich formula valid for any smooth function $v$ (see e.g., \cite{Ne,Re}):
$$2\!\!\int_{\partial \Omega^*}\!\!\!(x.\nabla v){\frac{\partial v}{\p n}} -\int_{\partial \Omega^*}\!\!\!(x.n)\vert \nabla v\vert ^2\!\!
=\!2\!\int_{\Omega^*}\!\! (x.\nabla v)\Delta v+(2-d)\!\!\int_{\Omega^*} \!\!\vert\nabla v\vert ^2. $$
For our eigenfunction $u$ in dimension 2, this yields
$$
2\!\!\int_{\partial \Omega^*}\!\!\!(x.\nabla u){\frac{\partial u}{\p n}} -\int_{\partial \Omega^*}\!\!\!(x.n)\vert \nabla u\vert ^2\!\!
= 0.
$$
Decomposing the gradient in its tangential and normal component, that is, $\nabla u=u_\tau \tau + u_n n$
and using $u_n=\sigma u$ yields
$$
2\sigma \int_{\partial \Omega^*} u u_\tau X.\tau + \sigma^2 \int_{\partial \Omega^*} u^2 X.n - \int_{\partial \Omega^*} u_\tau^2 X.n=0.
$$
Now, assume by contradiction that the eigenvalue is simple. For any deformation field $V$
the shape derivative of $D\sigma_K$ is nonpositive. Let us denote by $\mathcal{D}$
the set of diameter points (where the diameter is achieved, this is the projection 
of the set $\mathcal{D}_{\Omega^{\star}}$ introduced in Section \ref{secderdiam}). Let
$x\notin \mathcal{D}$ a point which does not belong to a diameter.
A small perturbation $V$ locally supported near $x$ does not change the diameter and since both $V$ and
$-V$ are admissible, we infer $d\sigma_k(\O^*,V)=0$ for all such $V$ which implies, according to
Theorem \ref{derisigma}:
\begin{equation}\label{mul1}
u_\tau^2-\sigma^2 u^2 -H \sigma u^2 =0 \quad\mbox{at any $x\notin \mathcal{D}$}.
\end{equation}

 We are going to distinguish two cases:
\begin{itemize}
\item {\bf Case 1.} The set $\mathcal{D}$ is discrete. 
In that case, the relation \eqref{mul1} holds everywhere on the boundary by continuity of the function and its
derivatives up to the boundary.
\item {\bf Case 2.} Now, let us assume that the set $\mathcal{D}$ (where the diameter is achieved) is not discrete. Let us consider a point $x\in \mathcal{D}$. If we perform a local perturbation driven
by $V$ near $x$, we can consider two situations: either $V.n$ is inward and then the diameter does
not change and we recover $d\sigma_k(\O^*,V) \leq 0$ at such a point or $V.n$ is outward and the
derivative of the diameter is positive, according to formulae \eqref{derdiam} which implies
$$D(\O^*) d\sigma_k (\O^*,V) \leq d(D \sigma_k)(\O^*,V) =\sigma_k(\O^*) d D  (\O^*,V) + D(\O^*) d\sigma_k(\O^*,V)  \leq 0$$
and then in any case, we have $d\sigma_k(\O^*,V) \leq 0$. Since this property holds for any $V$, we still
infer $d\sigma_k(\O^*,V) = 0$ and the equality \eqref{mul1} holds true on the whole boundary.
\end{itemize}
Therefore, in both cases, we have
$$
u_\tau^2-\sigma^2 u^2 = H \sigma u^2 \mbox{ on $\p\O^*$}.
$$
Thus, 
\begin{equation}\label{rel4}
2\sigma \int_{\partial \Omega^*} u u_\tau X.\tau = \int_{\p\O^*} u_\tau^2-\sigma^2 u^2 X.n ds= \sigma \int_{\p\O^*} H u^2 X.n \,ds\,.
\end{equation}
Now let us compute the left-hand side of \eqref{rel4} integrating by parts. Since 
$X.\tau = xx'+yy'$ and $2u u_\tau = \frac{d}{ds}\, u^2$ we have
\begin{equation}\label{rel5}
2\sigma \int_{\partial \Omega^*} u u_\tau X.\tau = - \sigma \int_{\partial \Omega^*} u^2 [xx"+yy"+{x'}^2+{y'}^2] ds\,.
\end{equation}
Now since $\frac{d}{ds} \tau =-H n$ and ${x'}^2+{y'}^2=1$, \eqref{rel5} provides
$$2\sigma \int_{\partial \Omega^*} u u_\tau X.\tau = \sigma \int_{\partial \Omega^*} H u^2 X.n ds - \sigma
\int_{\partial \Omega^*} u^2 ds$$
which, together with \eqref{rel4} would give
$$\sigma \int_{\partial \Omega^*} u^2 ds =0$$
a contradiction since $\sigma$ cannot be zero for the maximizer. This finishes the proof of 
the unconstrained case.

\medskip
Now let us consider the case with a convexity constraint. On strictly convex parts, we can perform any deformation and then the identity $u_\tau^2-\sigma^2 u^2 = H \sigma u^2$ still holds true on strictly
convex parts of the boundary. The flat parts (or segments) require more attention.
Let us consider such a segment, say $\Sigma \subset \p\O^*$. In the spirit of \cite{He-Ou} (see also
\cite[Theorem 4.2.2]{H1}) we can prove the following:

\begin{lemma}
Let $\O^*$ be a smooth maximizer of $D(\O)\sigma_k(\O)$ among convex sets and let $\Sigma$ be a segment
of extremities $A$ and $B$ included in the boundary of $\O^*$. Let $t\in [a,b]$, a parametrization 
of the segment (the
boundary is assumed to be oriented in the clockwise sense). Then, 
there exists a nonnegative-valued function
 $w$ defined on $[a,b]$ with triple roots at $a$ and $b$, such that
\begin{equation}\label{mul3}
u_\tau^2-\sigma^2 u^2 =  w''(t)\,.
\end{equation}
\end{lemma}

{\it Proof of the Lemma}: The diameter constraint cannot be achieved at any point of the segment
(but possibly its extremities). Thus, we just need to look at the derivative of the Steklov eigenvalue
which is given by Theorem \ref{derisigma}:
\begin{equation}\label{4.7}
d\sigma_k(\O^*,V)=\int_{\p\O^*} [u_\tau^2-\sigma^2 u^2 -H\sigma u^2] V.n \,ds.
\end{equation}
Recall that the the curvature $H=0$ on $\Sigma$. In formula \eqref{4.7}, the only perturbations $V$ which are
allowed are such that the deformed domain $(Id+\tau V)(\O^*)$ is
still convex (for small $\tau$). This holds true if and only if
$t\mapsto V.n(t)$ is a concave function on $[a,b]$. Let us denote
by $v=V.n$ such a concave function. Replacing in \eqref{4.7} and
using the relation  $u_\tau^2-\sigma^2 u^2 -H\sigma u^2=0$ which holds on the strictly convex parts
as explained above, yields on the segment $\Sigma$:
\begin{equation}\label{2a}
\int_a^b [u_\tau^2-\sigma^2 u^2] v\,dt\leq 0\,.
\end{equation}
Setting $w_2(t):=u_\tau^2-\sigma^2 u^2$ this can also be
rewritten
\begin{equation}
\label{3a}
\int_a^b w_2(t)v(t)\,dt\leq 0.
\end{equation} 
The latter estimate \eqref{3a} must be true for every (regular) concave function $v$.
In particular, in the case $v(t)=1$ and $v(t)=t$, both functions
$v$ and $-v$ are concave, therefore
\begin{equation}
\label{4a}
\int_a^b w_2(t)\,dt= 0\quad \int_a^b tw_2(t)\,dt=
0\;.
\end{equation} 
Now, let us introduce the functions defined by
$$w_1(t)=\int_a^t w_2(s)\,ds\quad\mbox{and}\quad w(t)=\int_a^t
w_1(s)\,ds=\int_a^t (t-s)w_2(s)\,ds\;.$$ According to \eqref{4a},
we have $w_1(a)=w_1(b)=w(a)=w(b)=0$. Integrating twice by parts,
it comes  
$$
\int_a^b w_2(t)v(t)\,dt=\int_a^b w(t)v''(t)\,dt.
$$ 
This last integral must be nonpositive (according to \eqref{3a}) for
every function $v$ concave,  i.e., for every smooth function $v$ such that
 $v''\leq 0$. This guarantees that $w\geq 0$. At last $a$ and $b$
are triple roots of  $w$ because $w''(a)= w_2(a)=0$ by continuity
of the gradient. This finishes the proof of the Lemma. \qed

\medskip
Let us come back to the proof of Theorem \ref{double}. We have already seen that 
$u_\tau^2-\sigma^2 u^2 = H \sigma u^2$ on the strictly convex parts of $\p\O^*$.
Now, let us consider a segment $\Sigma$. On such a segment $X.n$ is constant (equal to the distance, 
say $\delta$, of the origin to the line supporting the segment). Therefore, according to \eqref{mul3},
we have
$$
\int_\Sigma [u_\tau^2-\sigma^2 u^2] X.n dt = \delta \int_a^b w''(t) dt =w'(b)-w'(a) =0 =
\sigma \int_\Sigma H \sigma
u^2 X.n ds.
$$
Therefore, the relation
$$\int_{\p\O^*} u_\tau^2-\sigma^2 u^2 X.n ds= \sigma \int_{\p\O^*} H u^2 X.n \,ds$$
holds true on the whole boundary and we can conclude as in the case without convexity constraint.
\end{proof}
\begin{remark}
The numerical simulations of the next section suggest that the optimal domain is not exactly $C^3$ regular. It seems that its boundary has two singular points where the diameter is achieved.
Nevertheless, it is straightforward to check that Theorem \ref{double} remains true if we replace the $C^3$ regularity assumption by the following weaker assumptions that could be true for our optimal domains:
\begin{itemize}
\item the boundary of the optimal domain is $C^3$ except at a finite number of points;
\item the curvature $H$ is bounded;
\item the eigenfunction $u$ belongs to $C^1(\overline{\Omega})$.
\end{itemize}
\end{remark}

\section{Numerical simulations}\label{secnum}

In Section \ref{secqua}, we showed among other things that the disk is never a local maximizer of $\sigma_k(\Omega)$ under a diameter constraint. This leads us to provide some numerical computations in order to find some approximations of these maximizers in the plane. We point out that the diameter constraint is difficult to handle in a numerical point of view: this comes from the fact that on regions where this constraint is saturated not all arbitrarily small perturbations are admissible. 

A good tool for investigating the diameter constraint in the convex setting is the support function. 
This is why, in a first stage we consider the maximization problem in the class of convex sets. The support function of a set $\Omega \subset \mathbb{R}^2$ is defined for each $\theta \in [0,2\pi]$ by
\[ p(\theta) = \max_{x \in \Omega} x \cdot (\cos \theta, \sin \theta),\]
where the dot $\cdot$ denotes as usual the Euclidean scalar product. An intuitive interpretation of $p(\theta)$ is the distance from the origin to the tangent orthogonal to $\theta$ (see Figure \ref{fig:supp_func} for an illustration). With this geometric meaning of the support function in mind, it is obvious that the diameter or the width of $\Omega$ in the direction $\theta$ is given by $p(\theta)+p(\theta+\pi)$. 

\begin{figure}[h!]
	\centering
	\includegraphics[width=0.4\textwidth]{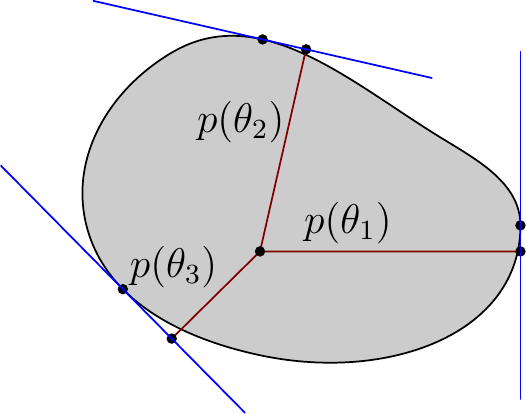}
	\caption{Geometric interpretation of the support function.}
	\label{fig:supp_func}
\end{figure}

The support function has been successfully used in the approximation of optimal shapes under convexity constraint  (see, e.g., \cite{BHcw12,BHL,AB18}). In the paper \cite{BHcw12}, the authors investigate various functionals related to the volume and perimeter. The work \cite{BHL} is devoted to the case of Dirichlet-Laplace eigenvalues and in \cite{AB18} the method is extended to the dimension three and multiple case tests are provided. If $p$ is the support function of a strictly convex domain $\Omega$ then, as recalled in these works, a parametrization of $\partial \Omega$ is given by
\begin{equation} \begin{cases}
x(\theta) = p(\theta) \cos \theta - p'(\theta) \sin \theta, \\
y(\theta) = p(\theta) \sin \theta + p'(\theta) \cos \theta.
\end{cases}
\label{eq:param_supp_func}
\end{equation}
This shows immediately that the radius of curvature is $\rho = p+p''$ and the convexity of $\Omega$ implies $p+p''\geq 0$. Conversely, a classical fact recalled in all these cited works is the fact that a $C^1$ and $2\pi$ periodic real function $p$ which satisfies (in the sense of distributions) $p+p'' \geq 0$ is the support function of a unique convex shape $\Omega \subset \mathbb{R}^2$.

In all the works cited above, the support function is discretized with the help of a truncated spectral decomposition (namely, Fourier series in dimension two and spherical harmonics decomposition in dimension three). This makes easier the treatment of some constraints, like the constant width constraint, but forces the support function to be smooth. On the other hand, as already seen in \cite{He-Ou}, when dealing with spectral functionals under convexity constraints, minimizers often tend to contain segments in their boundary. In such cases, the support function is not smooth anymore and then its parametrization through spectral decomposition is no longer appropriate. This is why in the following we choose a different approach which can handle discontinuities in the derivative of the support function. 

Consider $N$ a positive integer and $\theta_i = 2\pi i/N$, $i=0,...,N-1$ angles in $[0,2\pi]$. Then, the support function will be discretized by considering its values $p_i = p(\theta_i)$, $i=0,...,N-1$ at the angles chosen before. The first and second derivatives of $p$ are approximated using finite differences in the following way:
\begin{equation}
p'(\theta_i) = \frac{p_{i+1}-p_{i-1}}{2h} \quad \text{and} \quad p''(\theta_i) = \frac{p_{i+1}+p_{i-1}-2p_i}{h^2},
\label{First-Second_Deriv}
\end{equation}
for $i= 0,...,N-1$ (indices considered modulo $N$) and $h = 2\pi/N$.

The computation of the Steklov eigenvalues is done using the software FreeFEM (\cite{freefem}) while the constrained optimization is done with the algorithm IPOPT (\cite{IPOPT}). In the FreeFEM software the domain is meshed and finite elements are used in the computations. The main components of the optimization algorithm are shown below.

{\bf Convexity and Diameter constraints.} The convexity constraint is imposed pointwise for each $\theta_i$, that is,
$$
p(\theta_i)+p''(\theta_i) \geq 0 \quad i=0,...,N-1.
$$
Taking into account the second equality of \eqref{First-Second_Deriv}  then yields
\[ p_i+ \frac{1}{h^2}(p_{i+1}+p_{i-1}-2p_i) \geq 0,\ i=0,...,N-1,\] 
which can be translated into a set of $N$ linear inequality constraints on the variables $p_i$, $i=0,...,N-1$.

We have seen before that diameter constraints can be imposed by controlling the quantity $p(\theta)+p(\theta+\pi)$. In practice, we consider $N$ even, so that $\theta_i+\pi = \theta_{i+N/2}$ with indices considered modulo $N$. The fact that the shape has diameter at most $d$ is expressed by
\[ p_i+p_{i+N/2} \leq d,\ i=0,...,N/2-1.\]
In order to have diameter exactly equal to $d$, we impose the reverse inequality for one pair of opposite points:
\[ p_0+p_{N/2} \geq d.\]
At last, we obtain a set of $N/2+1$ linear inequality constraints. 

{\bf Construction of the mesh.} The inputs of the objective function are values of $p_0,...,p_{N-1}$. Starting from these values and using \eqref{eq:param_supp_func} we can find points $Q_i(x(\theta_i),y(\theta_i))$ by approximating the derivatives $p'(\theta_i)$ using centered finite differences as shown above (see \eqref{First-Second_Deriv}). The points $Q_i$ form a polygonal line whose interior is meshed in FreeFEM.

Note that the meshing algorithm in FreeFEM will give an error if the polygonal line contains self-intersections. In case such an error appears we reject the current computation. The algorithm IPOPT which deals with the optimization will eventually produce admissible vectors when imposing the convexity constraints shown above. 

The discretization points may be close on the boundary of $\partial \Omega$, especially close to eventual angular points. On the other hand, on parts which are almost flat, the discretization points will be rather sparse. In order to have a good finite element approximation the quality of the resulting mesh is improved using the command \texttt{adaptmesh} with parameters \texttt{hmax}=0.05*D, \texttt{nbvx}=50000, \texttt{iso}=1 refering to the maximal size of triangles, maximal number of vertices and the quality of the mesh. For more details, one should consult the FreeFEM documentation. 

{\bf Eigenvalue problem and gradient of the objective function.} Once the mesh is constructed, FreeFEM allows us to solve the eigenvalue problem starting from the variational formulation using finite elements. It is possible to recover the approximate eigenvalue and the associated eigenfunction. Concerning the finite element setup, $P_2$ finite elements are used for solving the eigenvalue problem and $P_1$ elements are used for evaluating the shape derivative (which contains derivatives of $P_2$ functions).  In the discrete setting, the eigenvalue $\sigma_k(\Omega)$ is a function of the parameters $p_i$, $i=0,...,N-1$:
\[ \sigma_k(\Omega) \approx F_k(p_0,...,p_{N-1}).\]
In order to use a gradient based optimization algorithm, it is necessary to compute the gradient of $F_k$ with respect to each one of the parameters. The classical method to handle this is to use the shape derivative formula given in Theorem \ref{derisigma}. Then, for each one of the parameters $p_i$, we look at the boundary perturbation $V_i$ obtained when considering perturbations $p_i+\delta t$ as $\delta t \to 0$. It suffices to put the perturbation $V_i$ in the shape derivative formula to obtain the gradient with respect to the variable $p_i$. A straightforward computation shows that a perturbation of the form $p_i+\delta t$ induces a vector field $V_i$ such that $V_i.n$ is equal to $1$ at $Q_i$ and is $0$ for every other point in the discrete boundary. Define $\chi_i$ to be a function which is piecewise affine on the segments $Q_iQ_{i+1}$ and which is $1$ at $Q_i$ and $0$ at $Q_j$ for $j \neq i$. Then the gradient of $F_k$ with respect to $p_i$ is approximated by
\[ \frac{\partial F_k}{\partial p_i} = \int_{\partial \Omega} 
  \left( 
  |\nabla u_k|^2 -(\partial_n u_k)^2- H \sigma_k u_k^2 
  \right) \chi_i d\sigma 
\]
where, as usual, $\sigma_k$, $u_k$ denote the $k$-th eigenvalue and associated eigenfunction and $H$ denotes the curvature. The FreeFEM command \texttt{curvature} is used to approximate the discrete curvature of the polygonal line.

{\bf Optimization algorithm.} As already mentioned before, the optimization is done in FreeFEM using the algorithm IPOPT. The inputs are the function $F_k$ and its gradient, as well as the matrices involving the linear discrete constraints associated to the convexity and diameter constraints. In addition to the linear constraints, pointwise positivity constraints are imposed on $p_i$, since we can assume that the origin is strictly inside our shape. The discretization uses $N=200$ angles in $[0,2\pi]$ and the diameter is fixed to $D=2$.

{\bf Results and remarks.} The algorithm is run for $1\leq N \leq 7$ 
and the resulting numerical optimal shapes are represented in Figure \ref{fig:results}. The numerical results give rise to the following remarks:
\begin{itemize}
	\item As predicted by the theoretical results, in each case the optimal eigenvalue is multiple: $\sigma_k(\Omega^*_k) = \sigma_{k+1}(\Omega_k^*)$
	\item In all the numerical results obtained the convexity constraint is saturated in some region, giving rise to segments in the boundary. Note that the direct discretization of the support function proposed here manages to properly capture this phenomenon, which was not the case for the Fourier decomposition used in \cite{AB18}.
	\item The diameter constraint seems to be saturated only at the two antipodal points included in the constraints. Moreover, angular points seem to be present at these antipodal points.
	\item The sequence of maximizers seems to become more and more flat as the index grows.
We discuss that point below. 
\end{itemize}

\begin{figure}[h!]
	\centering 
	\begin{tabular}{>{\centering\arraybackslash}m{0.3\textwidth}>{\centering\arraybackslash}m{0.3\textwidth}>{\centering\arraybackslash}m{0.3\textwidth}}
		\includegraphics[width=0.29\textwidth]{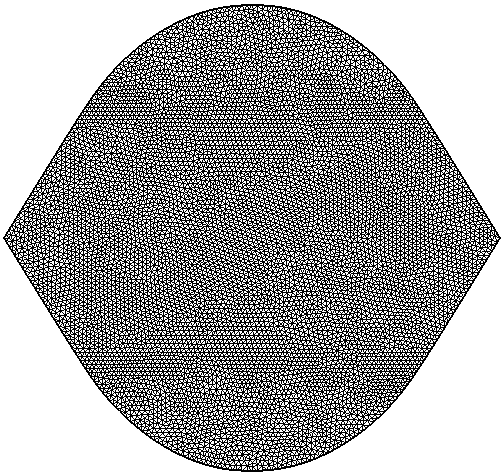} &
     	\includegraphics[width=0.29\textwidth]{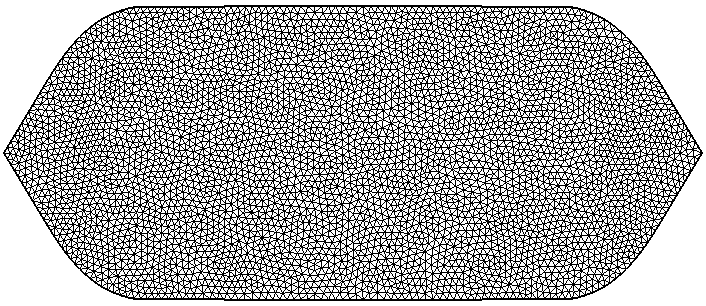} &
	    \includegraphics[width=0.29\textwidth]{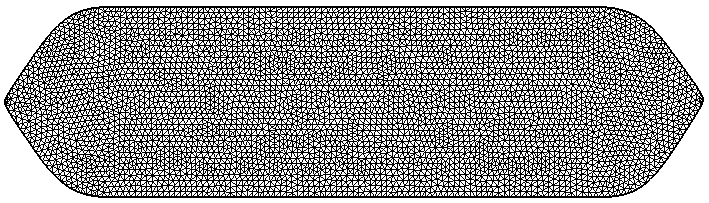}  \\
	    $\sigma_1(\Omega_1^*)D(\Omega_1^*)=2.13536$ &
	    $\sigma_2(\Omega_2^*)D(\Omega_2^*)=4.73269$ &
	    $\sigma_3(\Omega_3^*)D(\Omega_3^*)=7.33378$ 
    \end{tabular}
    \begin{tabular}{>{\centering\arraybackslash}m{0.45\textwidth}>{\centering\arraybackslash}m{0.45\textwidth}}
		\includegraphics[width=0.4\textwidth]{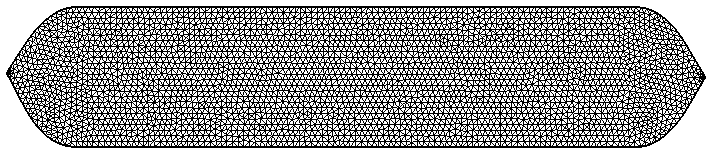} &
		\includegraphics[width=0.4\textwidth]{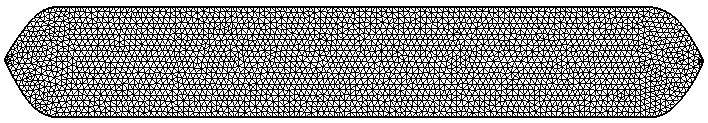} \\
		$\sigma_4(\Omega_4^*)D(\Omega_4^*)=9.96641$ & 
	  $\sigma_5(\Omega_5^*)D(\Omega_5^*)=12.5721$ \\
 		\includegraphics[width=0.4\textwidth]{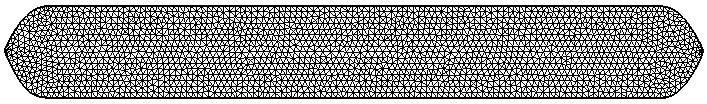} &
		\includegraphics[width=0.4\textwidth]{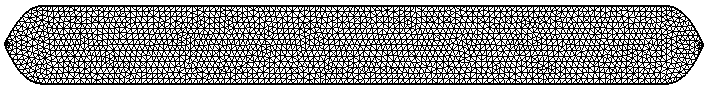} \\
	$\sigma_6(\Omega_6^*)D(\Omega_6^*)=15.1812$ & $\sigma_7(\Omega_7^*)D(\Omega_7^*)=17.8068$
	\end{tabular}
\caption{Results of the optimization algorithm: maximization of $\sigma_k(\Omega)D(\Omega)$ under convexity constraint.}
\label{fig:results}
\end{figure}

\begin{remark}
The numerical results above suggest that the sequence of optimal
domains converges to a segment. This would be really interesting since it would
differ with the case of Dirichlet eigenvalues where it is shown in \cite{BF}
(perimeter constraint) and in \cite{BHL} (diameter constraint) that the sequence of 
optimal domains converges to a disk (or even a ball in any dimension in the second case).
A possible strategy to prove that fact can be by contradiction:
\begin{enumerate}
\item Assume that the sequence of optimal domains $(\Omega_k^*)_k$ converges for the Hausdorff metric to
a convex open set $\Omega_\infty$ with diameter $d_0$.
Then, provide a uniform control of the difference of eigenvalues of the kind 
$0\leq \sigma_k(\Omega_k^*) - \sigma_k(\Omega_\infty) \leq \varepsilon(k)$
with $\varepsilon(k)/k \to 0$ as $k\to \infty$.
\item Now fix a smooth convex set $\omega$ (like an elongated ellipse) of diameter $d_0$ and with a perimeter such that $P(\omega)< P(\Omega_\infty)$. 
Optimality of $\Omega_k^*$ then yields
$$\sigma_k(\Omega_\infty)+  \varepsilon(k) \geq \sigma_k(\Omega_k^*) \geq \sigma_k(\omega).$$
\item Finally apply Weyl's law
for Steklov eigenvalues that writes (see \cite{GP1})
\begin{equation}\label{weyl}
\sigma_k(\Omega)\sim \frac{2\pi k}{P(\Omega)} \quad \text{as}\:\:k\to \infty
\end{equation}
to get $P(\Omega_\infty)\leq P(\omega)$, which is a contradiction.
\end{enumerate}
The flaw of this strategy is that Weyl's law \eqref{weyl} is only known for smooth open sets,
see the discussion in \cite{GP1}. Now, we have no guarantee that the limit convex domain
$\Omega_\infty$ would be smooth!
\end{remark}

{\bf Alternative approach and the non-convex case.} The method described above imposed rigorously the convexity and diameter constraints. As can be seen in Figure \ref{fig:results} the numerical maximizers seem to saturate the diameter constraint at exactly two points. This suggests \emph{a posteriori} that a simpler parametrization should work. Moreover, it seems to be enough to only impose the diameter condition for exactly two points.

One may consider the segment $[-D/2,D/2]\times\{0\}$ in $\Bbb{R}^2$ and the family of shapes defined as regions contained between the graphs of two functions $f_1,f_2 : [-D/2,D/2] \to \mathbb{R},\ f_1\leq f_2$. From a discrete point of view $f_1$ and $f_2$ are discretized at an equidistant family of points in $[-D/2,D/2]$ with values $p_1,...,p_N$ and $q_1,...,q_N$. The convexity of $f_1$ and the concavity of $f_2$ translate to the discrete inequalities
\[ \frac{f_{i-1}+f_{i+1}}{2} \geq f_i,\ \frac{q_{i-1}+q_{i+1}}{2} \leq q_i,\ i=1,...,N\]
with the convention $p_0=p_{N+1}=q_0=q_{N+1}=0$. Given values $p_i\leq q_i,\ i=1,...,N$, the discrete domain is meshed in FreeFEM and the Steklov eigenvalue problem is solved using finite elements as before. The computation of the gradient with respect to the variables $p_i,q_i$ is similar to what was done with the support function. One only needs to keep in mind that a perturbation in these variables amounts to a perturbation in the $y$ direction of the normal to the boundary of $\Omega$.

The resulting numerical algorithm gives exactly the same results as those shown in Figure \ref{fig:results}. Moreover, even if the diameter constraint is not imposed during the optimization at other points than the endpoints of the segment $[-D/2,D/2]$, the numerical shapes obtained verify the diameter constraint everywhere.

This alternative method has the advantage that it can also handle the non-convex case. Indeed, if we do not impose that $f_1$ is convex and $f_2$ is concave during the optimization process we obtain the non-degenerate shapes shown in Figure \ref{fig:results_non_convex}. Note that for the first eigenvalue, the result is a slight loss of convexity near the two corners observed in the domain. However, the corresponding maximal eigenvalue is only a bit larger than the one obtained imposing the convexity constraint. For $k\in \{2,3\}$ we observe obvious departs from the convexity near the parts where the results in Figure \ref{fig:results} contained segments in the boundary. One may note similarities between the maximizer of $\sigma_2(\Omega)D(\Omega)$ and the maximizer of $\sigma_2(\Omega)$ under area constraint shown in \cite{BBG}, but the case $k=3$ is completely different.

 The same remarks as in the convex case hold: the $k$-th eigenvalue is multiple at the optimum, the diameter constraint is saturated at exactly two points and the minimizers become flatter as $k$ grows. The fact that the numerical algorithm does find non-degenerate shapes suggests that the existence of a maximizer should hold even without the convexity assumption.

\begin{figure}
	\centering
	\begin{tabular}{m{0.3\textwidth}m{0.3\textwidth}m{0.3\textwidth}}
	\includegraphics[width=0.29\textwidth]{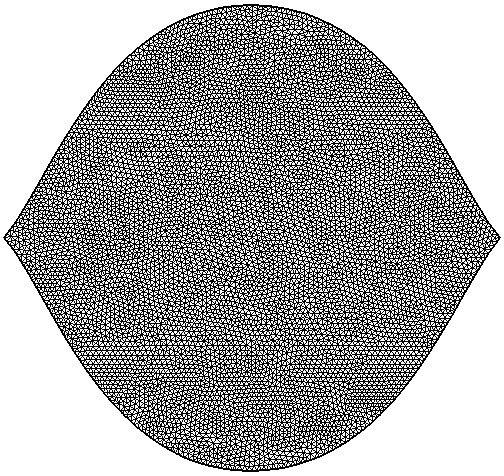}&
	\includegraphics[width=0.29\textwidth]{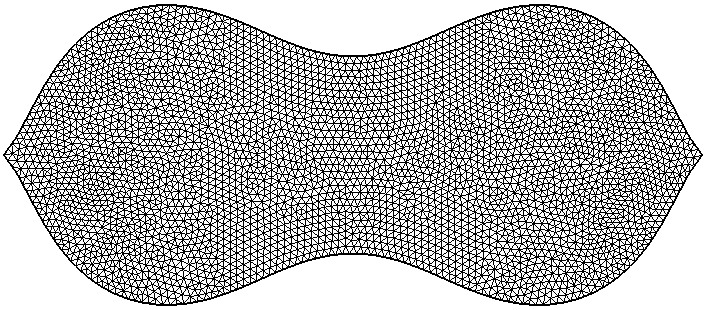}&
	\includegraphics[width=0.29\textwidth]{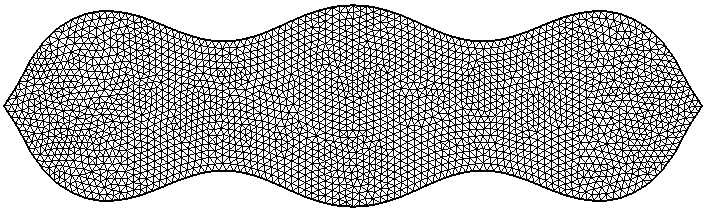} \\
   \centering	$\sigma_1(\omega_1^*)D(\omega_1^*) = 2.13623$  & \centering$\sigma_1(\omega_2^*)D(\omega_2^*) = 4.92925$  & \centering$\sigma_1(\omega_3^*)D(\omega_3^*) = 7.76108$
	\end{tabular}
\caption{Numerical results obtained using the alternative method for $k\in \{1,2,3\}$: maximization of $\sigma_k(\Omega)D(\Omega)$ without the convexity constraint.}
\label{fig:results_non_convex}
\end{figure}

\begin{center}
	\textsc{Acknowledgments}
\end{center}

This work was partially supported by the project ANR-18-CE40-0013 SHAPO financed by the French Agence Nationale de la
Recherche (ANR).

\medskip

Abdelkader \textsc{Al Sayed}, Institut \'Elie Cartan de Lorraine, UMR 7502, Universit\'e de Lorraine CNRS, email: \texttt{alsayed.abdkader@gmail.com }

Beniamin \textsc{Bogosel}, Centre de Math\'ematiques Appliqu\'ees, Ecole Polytechnique, UMR CNRS 7641, email: \texttt{beniamin.bogosel@polytechnique.edu}

Antoine \textsc{Henrot}, Institut \'Elie Cartan de Lorraine, UMR 7502, Universit\'e de Lorraine CNRS, email: \texttt{antoine.henrot@univ-lorraine.fr} (corresponding author)

Florent \textsc{Nacry}, Laboratoire de Math\'ematiques, Physique et Syst\`emes, Universit\'e de Perpignan Via Domitia, 52 Avenue Paul Alduy, 66860 Perpignan, email: \texttt{florent.nacry@univ-perp.fr}

\end{document}